\documentclass[11pt]{elsarticle}
\usepackage[all,pdf]{xy}

\usepackage{amsmath,amssymb,amsthm}
\usepackage{hyperref}
\usepackage{mathrsfs}
\usepackage{graphicx}
\usepackage[all,pdf]{xy}
\usepackage{tikz}
\usetikzlibrary{cd}
\usepackage{titlesec}

\theoremstyle{plain}
\newtheorem{theorem}{Theorem}[section]
\newtheorem{lemma}[theorem]{Lemma}

\newtheorem{corollary}[theorem]{Corollary}

\theoremstyle{definition}
\newtheorem{definition}[theorem]{Definition}

\newcommand{\rom}[1]{\rm{\uppercase\expandafter{\romannumeral #1}}}

\makeatletter
\def\ps@pprintTitle{%
  \let\@oddhead\@empty
  \let\@evenhead\@empty
  \def\@oddfoot{\reset@font\hfil\thepage\hfil}
  \let\@evenfoot\@oddfoot
}
\makeatother

\begin{document}

\begin{frontmatter}

\title{Stone duality of Lawson compact algebraic L-domain}
\tnotetext[t1]{This work is supported by  the Natural Science Foundation of Shandong Province, China (No. ZR2024QA195) and the Natural Science Foundation of China (No.12501645)}
\author{Huijun Hou}
\address{School of Mathematical Sciences, Qufu Normal University, Qufu, Shandong, 273165, China}
\ead{houhuijun2021@163.com}
\author{Ao Shen\corref{a1}}
\address{School of Mathematical Science, Tiangong University, Tianjin, 300387, China}
\cortext[a1]{Corresponding author.}
\ead{shenao2020@163.com}
\begin{abstract}
In this paper, a subclass of bounded distributive lattices, that is, finitely disjunctive distributive lattices (FDD-lattices) have been introduced. Then we apply it to establish a Stone duality for Lawson compact algebraic L-domains. Furthermore, we develop a dual equivalence between the category of  $FDD$-lattices with lattice homomorphisms and that of Lawson compact algebraic $L$-domains with  spectral maps.

\end{abstract}

\end{frontmatter}

\section{Introduction}
Duality theory was originated by Stone \cite{Stone1936} who first connected topological spaces and algebras. He developed the dual equivalence between  the category of Boolean algebras with Boolean homomorphisms and that of Stone spaces with continuous maps. Based on the duality, he \cite{Stone1937}  obtained the Stone's 
generalization of the Weierstrass approximation theorem and Stone-\v{C}ech compactification. After that, Stone \cite{Stone1938} extended this duality from Boolean algebras to bounded distributive lattices by  spectral spaces.  Concretely, he showed that the set $Spec(L)$ of prime filters of a bounded distributive lattice $L$ endowed with hull-kernel topology is a spectral space; and conversely, the set $CO(X)$ of clopen sets of a spectral space $X$ is a bounded distributive lattice. 
The Stone duality has been generalized to many ordered structures, for instance, to posets  \cite{Gonzalez2016}, to  distributive posets  \cite{David1992}, to  semilattices \cite{Celani2020}, to  bounded lattices \cite{Urquhart1978}, to  distributive join-semilattices \cite{Gratzer1978} and so on. 

Indeed, the Stone duality for bounded distributive lattice has a close relation with domain theory. Combining with  the one-to-one correspondence between  bounded distributive lattices and arithmetic frames with 1 being a compact element,  a special case of Nachbin's result \cite{Nachbin1949} that the one-to-one correspondence between  semilattices and algebraic lattices, we know that the category of arithmetic frames with 1 being a compact element is dually equivalent to that of spectral spaces. (See \cite[Exercise  \uppercase\expandafter{\romannumeral5}-5.22]{Gierz2003}). Moreover, Isbell \cite{Isbell1972} constructed the dual equivalence between spatial frames and sober spaces. Hofmann and Lawson \cite{Hofmann1978} showed that the category of continuous frames with frame homomorphisms is dually equivalent to that of locally compact sober spaces with continuous maps. Recently, Shen and Li \cite{Shen2024}  provided a Stone duality for continuous lattices.

The Jung-Tix problem \cite[]{Jung1998}, a central question in domain theory,  asks how to find  a subcategory consisting of dcpos and Scott
continuous maps that is cartesian closed
and closed under the probabilistic powerdomain construction?
In recent years, numerous scholars have proposed various approaches aiming to tackle this problem, among which the Lawson compactness has played a significant role. This paper will investigate Lawson compact algebraic $L$-domains and  present a Stone duality for them. Concretely, we introduce finitely disjunctive distributive lattices. Then we show that there is a one-to-one correspondence between  $FDD$-lattices and  Lawson compact algebraic $L$-domains. Focusing on the category of  $FDD$-lattices with lattice homomorphisms, we find the category of 
Lawson compact algebraic $L$-domains with spectral maps and obtain a dual equivalence between them.


\section{Preliminaries}

In this section, we will give a brief introduction to the concepts and notions of topology and domain theory that required in this paper.
One can  refer to \cite{Gierz2003} for more details.

Given a poset $P$, $A\subseteq P$ is an \emph{upper set} if $A = {\uparrow} A$, where ${\uparrow} A = \{x\in P: x\geq a\  \mathrm{for\  some}\ a\in A\}$. The  \emph{lower set} is defined dually. The subset $D$ of $P$ is \emph {directed}  if every finite subset  of $D$ has an upper bound in $D$. A subset $U$ of $P$ is called \emph{Scott open} iff it is an upper set and for each directed subset $D$ whose supremum exists, written as $\vee^{\uparrow} D$, $\vee^{\uparrow} D\in U$ implies $D\cap U\neq \emptyset$.  The collection of all Scott open subsets, denoted by $\sigma(P)$, form a topology  called \emph{Scott topology} and in general, we write $\Sigma P = (P, \sigma(P))$.   A continuous mapping between two posets equipped with the Scott topology is called \emph{Scott-continuous}. 
An element $x\in P$ is called \emph{compact} if for any directed subset $D\subseteq P$ with $\vee^{\uparrow} D$ existing and $x\leq \vee^{\uparrow} D$, there exists some $d\in D$ such that $x\leq d$. The subset of all compact elements  is denoted by $K(P)$.

An \emph{$\inf$ semilattice} is a poset $P$ in which any two elements $a, b$ have an $\inf$, denoted by $a\wedge b$.  A \emph{$\sup$ semilattice} is defined dually, and denoted by $a\vee b$. We call a poset be a \emph{lattice} if it is both an $\inf$ semilattice and a $\sup$ semilattice. An element $p\in P$ is called \emph{co-prime} if $p = 0$ or $P\setminus {\uparrow p}$ is an ideal, and $p\neq 0$ is co-prime iff  $p\leq x\vee y$ implies $p\leq x$ or $p\leq y$. Let $Cp(L)$ denote the set $\{p\in P: p\neq 0\; \rm{and}\; \rm{p\; is\; co{-}prime}\}$.
 A \emph{distributive lattice} $L$ is a lattice which satisfies the distributive law:
 \begin{center}
 	for any $a,b,c\in L$, $a\wedge(b\vee c) = (a\wedge b)\vee(a\wedge c)$.
 \end{center}
 If the distributive lattice $L$ also have the least element 0 and the greatest element 1, we shall say $L$ be a bounded distributive lattice.
 A \emph{complete lattice} is a poset in which every subset has a $\sup$ and an $\inf$. We say a dcpo $L$ is an \emph{algebraic domain} if every $x$ in $L$ satisfies that $\downarrow x\cap K(L)$ is directed and $x = \vee^{\uparrow} (\downarrow x\cap K(L))$. An algebraic domain in which every principal ideal $\downarrow x$ is a complete
lattice (in its induced order) is called an \emph{algebraic L-domain}.

Let $X$ be a topological space and $\mathcal O(X)$ be the set of all open subsets of $X$.  A subset $K$ of $X$ is compact in $\mathcal O(X)$ if and only if for each directed subset $\mathcal D\in \mathcal O(X)$ with $K\subseteq \bigcup \mathcal D$, there is a $U\in \mathcal D$ containing $K$. 
We call  a subset $U\in  \mathcal O(X)$ compact-open if it is compact and open. We use $CO(X)$ to denote the set of all compact-open subsets of $X$.

\begin{definition}{\rm(\cite{Gierz2003})}
Let L be a poset. We call the topology generated by the complements $L\setminus {\downarrow x}$ of principal filters (as subbasic open sets) the \emph{lower topology} and denote it by $\omega(L)$.
\end{definition}

\begin{definition}{\rm(\cite{Gierz2003})}
	Let $L$ be a dcpo. Then the common refinement of $\sigma (L)\vee \omega(L)$ of the Scott topology and the lower topology is called the \emph{Lawson topology}.
\end{definition}

\begin{definition}{\rm(\cite{Gierz2003})}
	An element $m$ of a poset $P$ is a \emph{minimal upper bound} (or
\emph{``mub''} for short) for a subset $A\subseteq P$ if $m$ is an upper bound for $A$ that is minimal in the set of all upper bounds of $A$. The poset $P$ is \emph{mub-complete} if given any finite subset $F\subseteq P$ and any upper bound $u$ of $F$ there exists a minimal upper
bound $y$ of $F$ such that $y\leq u$.
\end{definition}

  We usually represent the set of all minimal upper bounds of a finite subset $F$ with $mub(F)$.
\begin{theorem}{\rm(\cite{Gierz2003})}\label{mub}
	The following are equivalent for an algebraic domain $L$:
	\begin{enumerate}[(1)]
		\item $L$ is Lawson compact, i.e., $L$ is compact in the Lawson topology.
		\item $K(L)$ is mub-complete and every finite set of $K(L)$ has only
		finitely many minimal upper bounds.
	\end{enumerate}
\end{theorem}
\begin{lemma}{\rm(\cite{plo81})}\label{co}
	Let $L$ be an algebraic domain. Then any compact-open subset of $\Sigma L$ can be written as the form $\bigcup_{i\in F}{\uparrow k_{i}}$, where $k_{i}\in K(L)$ and $F$ is a finite subset of the natural numbers set.
	\end{lemma}
\begin{theorem}{\rm(\cite{ga})}\label{ald}
	An algebraic domain is an algebraic L-domain if and only if every bounded nonempty subset has a meet.
\end{theorem}

\section{Main Results}

\begin{definition}
	Let $L$ be a bounded lattice. For any $x, y\in L$, we say $x$ and $y$ are \emph{disjoint} if $x\wedge y = 0$.  An finite subset $F$ of $L$ is disjoint if any pair of different elements $x$ and $y$ in $F$ are disjoint. Moreover, the sup of $F$ is labeled as $\dot\vee F$.
\end{definition}

\begin{definition}\label{fddl}
	A bounded distributive lattice  $L$ is called a \emph{finitely disjunctive distributive lattice} (for short, $FDD$-lattice) if it also satisfies the following conditions:
	\begin{enumerate}[(1)]
		\item For each $x\in L$, there exists a finite subset $F\subseteq Cp(L)$ such that $x = \vee F$;
		\item For any $a, b\in Cp(L)$, $a\wedge b$ can be written as a join of finitely disjoint elements of $Cp(L)$.
	\end{enumerate}
\end{definition}

\begin{definition}\cite{Davey2002}
	Let $L$ be an $FDD$-lattice. A lattice homomorphism from $L$ to $\{0, 1\}$ is called a \emph{point} of $L$. We usually use $pt(L)$ to denote the poset composed of all points of $L$ with the pointwise order.
\end{definition}

%
Let $\mathbf{FFD}$ be the category of all $FDD$-lattices and lattice homomorphisms.
\begin{definition}\cite{Stone1938}
	Let $L$ and $M$ be two algebraic $L$-domain. A mapping $f: L\rightarrow M$ is called \emph{spectral} if for any $U\in CO(M)$, $f^{-1}(U)\in CO(L)$.
\end{definition}

Let $\mathbf{LCA}$  denote the category of all  Lawson compact algebraic $L$-domains  and spectral mappings. 

	\begin{definition}{\bf\cite{sa1987}}
	A space $X$ is \emph{coherent} if the compact-open subsets of $X$ form a basis closed under finite intersections, i.e. $CO(X)$ is a distributive sub-lattice of $\mathcal O(X)$.
\end{definition}

\begin{theorem}{\bf\cite{plo81}}\label{plo81}
	Any coherent algebraic domain $M$ is isomorphic to $pt(CO(M))$, where the isomorphism $\theta_M: M\rightarrow pt(CO(M))$ is defined as: 
	\begin{center}
		$\theta_{M}(x) = \vee^{\uparrow}\{\gamma_{{\uparrow}k}: k\in {\downarrow x}\cap K(M)\},	\forall x\in M$.
	\end{center} 
\end{theorem}
One can easily verify that each Lawson compact algebraic $L$-domain as a space is coherent. Thus the following result can be obtained:
\begin{corollary}
	Every Lawson compact algebraic $L$-domain $L$ is isomorphic to $pt(CO(L))$ under the isomorphism defined in Theorem \ref{plo81}.
\end{corollary}

\begin{theorem}
	If $L$ is a Lawson compact algebraic $L$-domain, then  $CO(L)$ is an $FDD$-lattice. 
	\begin{proof}
	$\mathbf{Claim\; 1}$: $CO(L)$ is a distributive lattice.
	
	Assume $U, V$ be compact-open subsets of $\Sigma L$.  According to the obvious fact that the union of arbitrary finite number of compact-open subsets is still compact-open, we know $U\cup V\in CO(L)$.  By Lemma \ref{co}, it is reasonable to assume that $U = \cup_{i\in F}{\uparrow}k_{i}$, $V = \cup_{j\in J}{\uparrow}k_{j}$, where $F$ and $J$ are finite subsets of natural numbers. Then
	\begin{align*}
		U\cap V 
		 & =
		 (\cup_{i\in F}{\uparrow}k_{i})\cap (\cup_{j\in J}{\uparrow}k_{j})  \\
		& =  \cup_{i\in F}\cup_{j\in J}({\uparrow}k_{i}\cap {\uparrow}k_{j}).
	\end{align*}
 Notice that for each $i\in F, j\in J$, ${\uparrow}k_{i}\cap {\uparrow}k_{j}$ is actually the set of all upper bounds of $\{k_{i}, k_{j}\}$. Thus  ${\uparrow}k_{i}\cap {\uparrow}k_{j} = {\uparrow}G_{i,j} $, where each $G_{i,j}$ is a finite subset of $mub(\{k_{i}, k_{j}\})$ in $K(L)$ by Lemma \ref{mub}.  This immediately reveals that $U\cap V$ can be written as a union of finite principal filters of compact elements, more precisely,  a union of finite compact-open subsets. So $U\cap V$ is compact-open. At this point, we know $CO(L)$ is a lattice ordered by inclusion, in which the  sup and inf of the finite elements are actually the  union and intersection of sets. And inherently, $CO(L)$ is a distributive lattice.  
 
$\mathbf{Claim\; 2}$:  The co-prime elements of $CO(L)\setminus\{\emptyset\}$  are precisely $\{{\uparrow}k: k\in K(L)\}$.
 
 On the one hand,  suppose $U$ is a co-prime element of $CO(L)$, Lemma \ref{co} enables us to write $U = \cup_{i\in F}{\uparrow}k_i$, where $F$ is a finite subset of natural numbers. Then by the property of being co-prime,  we can find an $i_{0}\in F$ satisfying that $U = {\uparrow}k_{i_{0}}$.
 On the other hand, for each $k\in K(L)$ with ${\uparrow}k\subseteq U\cup V$, where $U, V\in CO(L)$, then necessarily  $k\in U$ or $k\in V$. So we get that ${\uparrow}k\subseteq U$ or ${\uparrow}k\subseteq V$. This is equivalent to saying that all of $\{{\uparrow}k: k\in K(L)\}$ are co-prime elements of $CO(L)$.
	
$\mathbf{Claim\; 3}$:  $CO(L)$ is an $FDD$-lattice.
	
	Building upon the established result of Claim 2 and  invoking Lemma \ref{co} again, we derive  that each element of $CO(L)$ can be written as a union of finite elements of $Cp(CO(L))$. Thus   Condition (1) of Definition \ref{fddl} is satisfied.
	
	 Let ${\uparrow}l, {\uparrow}m$ be two co-prime elements of $CO(L)$. By Claim 1, we can write  ${\uparrow}l\cap {\uparrow}m = \cup_{i\in F}{\uparrow}k_{i}$, where $F$ is a finite subset of natural numbers. Moreover, this finiteness allows us to justifiably require that these sets $\{{\uparrow}k_{i}: i\in F\}$ are mutually incomparable under inclusion.  To finish the proof, it remains to show that for any two $i, j\in F$, ${\uparrow}k_{i}\cap {\uparrow}k_{j} = \emptyset$. Suppose  there exist $i, j\in F$   such that ${\uparrow}k_{i}\cap {\uparrow}k_{j} \neq \emptyset$. Then there is an element $c\in {\uparrow}k_{i}\cap {\uparrow}k_{j}$, that is, $c$ is an upper bound of $k_{i}, k_{j}$. Since $L$ is an  $L$-domain, $k_{i}\wedge k_{j}\in L$. Note that both $k_{i}$ and  $k_{j}$ are upper bounds of $\{l, m\}$. Thus $k_{i}\wedge k_{j}$ is also an upper bound of $\{l, m\}$, which means that there must be an $i_{0}\in F$ such that $k_{i}\wedge k_{j}\geq k_{i_{0}}$. Then we have $k_{i}\geq k_{i_{0}}, k_{j}\geq k_{i_{0}}$. Clearly, this contradicts the assumption that these finite sets  $\{{\uparrow}k_{i}: i\in F\}$ are pairwise incomparable. So Condition (2) of Definition \ref{fddl} holds. 
	 

	\end{proof}

\end{theorem}

\begin{theorem}\label{th37}
	If $L$ is an $FDD$-lattice, then $pt(L)$ is a Lawson compact algebraic $L$-domain.
	
	\begin{proof}
		$\mathbf{Claim\; 1}$:  $pt(L)$ is a dcpo.
		
		Assume $\{p_{i}: i\in I\}$ is a directed subset of $pt(L)$. We define $\vee^{\uparrow} p_{i}: L\rightarrow \{0, 1\}$ as
		\begin{center}
			$\vee^{\uparrow} p_{i}(x) = \vee^{\uparrow}\{p_{i}(x): i\in I\}$, for any $x\in L$.
		\end{center}
		One can easily verify that $\vee^{\uparrow} p_{i}$ preserves $0, 1$ and finite  nonempty sups. The remainder we need to prove is that $\vee^{\uparrow} p_{i}$ preserves  finite nonempty  infs. To do this, we pick $x, y\in L$ to show $\vee^{\uparrow} p_{i}(x\wedge y) = \vee^{\uparrow} p_{i}(x)\wedge \vee^{\uparrow} p_{i}(y)$. If $\vee^{\uparrow} p_{i}(x\wedge y) = 0$, then for any $i\in I$, $p_{i}(x\wedge y) = 0$. Suppose $\vee^{\uparrow} p_{i}(x) = 1$ and $\vee^{\uparrow} p_{i}(y) = 1$, there will exist $i_{1}, i_{2}\in I$ such that $p_{i_{1}}(x) = 1$ and $p_{i_{2}}(y) = 1$. By the directness of $\{p_{i}: i\in I\}$, we can find $i_{3}\in I$ such that $p_{i_{3}}(x) = 1, p_{i_{3}}(y) = 1$. Then from the fact that $p_{i_3}\in pt(L)$, we deduce that $p_{i_{3}}(x\wedge y) = p_{i_{3}}(x)\wedge p_{i_{3}}(y)  = 1$, which results in a obvious contradiction with $p_i(x\wedge y) = 0$ for all $i\in I$. So $\vee^{\uparrow} p_{i}(x) = 0$ or $\vee^{\uparrow} p_{i}(y) = 0$, that is to say, $\vee^{\uparrow} p_{i}(x)\wedge \vee^{\uparrow} p_{i}(y) = 0$. If  $\vee^{\uparrow} p_{i}(x\wedge y) = 1$, then there exists an $i_{0}\in I$ such that $p_{i_{0}}(x\wedge y) = 1$. Since $p_{i_{0}}\in pt(L)$, $p_{i_{0}}(x)\wedge p_{i_{0}}(y) = p_{i_{0}}(x\wedge y) = 1$, which means $p_{i_{0}}(x) =1$ and $p_{i_{0}}(y) = 1$. Thus $\vee^{\uparrow} p_{i}(x) = 1$ and $\vee^{\uparrow} p_{i}(y) = 1$. Naturally, we have $\vee^{\uparrow} p_{i}(x)\wedge \vee^{\uparrow} p_{i}(y) = 1$. Consequently, $\vee^{\uparrow} p_{i}(x\wedge y) = \vee^{\uparrow} p_{i}(x)\wedge \vee^{\uparrow} p_{i}(y)$ holds.
		
		Fixed an $a\in L$, we define a mapping $\gamma_{a}: L\rightarrow \{0, 1\}$ as following:
		\vspace{-1.5em}
		\begin{center}
			\[
			\gamma_{a}(x) = 
			\begin{cases}
				1, & \text{if} \;\;  x \geq a;  \\
				0, &  \text{otherwise}.
			\end{cases}
			\]
		\end{center}
Notably, for any $a\in Cp(L)$, the definition of $\gamma_a$ makes one immediately derive  that  $\gamma_a\leq p$ iff $p(a) = 1$.
			
		$\mathbf{Claim\; 2}$: For each $a\in Cp(L)$, $\gamma_{a}\in pt(L)$.
		
  One can easily verify that   $\gamma_{a}$ preserves $0, 1$ and all  finite nonempty infs. Now consider the sups of  finite  nonempty subsets. For any $x, y\in L$, we want to show $\gamma_{a}(x\vee y) = \gamma_{a}(x)\vee\gamma_{a}(y)$. If $\gamma_{a}(x\vee y) = 0$, then $x\vee y\ngeq a$. This means  $x\ngeq a$ and $y\ngeq a$. So $\gamma_{a}(x) = 0$ and $\gamma_{a}(y) = 0$, which immediately implies that $\gamma_{a}(x)\vee\gamma_{a}(y) = 0$. If $\gamma_{a}(x\vee y) = 1$, then $x\vee y\geq a$. By the co-primeness of $a$, we get that $x\geq a$ or $y\geq a$. Thus $\gamma_{a}(x) = 1$ or $\gamma_{a}(y) = 1$, directly from which we conclude that $\gamma_{a}(x)\vee\gamma_{a}(y) = 1$. In a word, $\gamma_{a}(x\vee y) = \gamma_{a}(x)\vee\gamma_{a}(y)$ holds.
  
$\mathbf{Claim\; 3}$:  For each $p\in pt(L)$, the set $\mathcal H_p = \{\gamma_a: \gamma_a\leq p,  a\in Cp(L)\}$ is directed.
	
	Firstly, we need to say that the set $\mathcal H_p$ is nonempty.  Since $p$ is a lattice homomorphism between two bound lattices, $p(1) = 1$. By the definition of an $FDD$-lattice, we can find  a finite subset $F$ of $Cp(L)$ such that $1 = \vee F$, which means $p(\vee F) = 1$, and further, $\vee \{p(x): x\in F\} = 1$. So there exists an $x_0$ satisfying that $p(x_0) = 1$. From which one can deduce that $\gamma_{x_0}\leq p$ and meanwhile, $\gamma_{x_0}\in \mathcal H_p$. Thus, $\mathcal H_p$ is nonempty.

	Now consider $\gamma_{a_{1}}, \gamma_{a_{2}}\in \mathcal H_p$.  Then $\gamma_{a_{1}}, \gamma_{a_{2}}\leq p$, which induces that	$p(a_1) = 1$ and $p(a_2) = 1$. Since $p\in pt(L)$, we have $p(a_1\wedge a_2) =  p(a_1)\wedge p(a_2) = 1$.
	As  $a_{1}, a_{2}\in Cp(L)$, by Definition \ref{fddl}, $a_1\wedge a_2 = \dot\vee F$, where $F$ is a  finite disjoint subset of $Cp(L)$.  Thus we further deduce that $p(\dot\vee F) = 1$. Then $\vee\{p(x): x\in F\} = 1$ utilizing $p\in pt(L)$ again. This means $p(x_0) = 1$ for some $x_0\in F$, which is equivalent to saying $\gamma_{x_0}\leq p$. The equation $a_1\vee a_2 = \dot\vee F$ indicates that $x_0\leq a_1\wedge a_2$, naturally, $x_0\leq a_1, a_2$, so $\gamma_{a_{1}}, \gamma_{a_{2}}\leq \gamma_{x_0}$.   Hence, we find a $\gamma_{x_0}\in \mathcal H_p$ that is larger than $\gamma_{a_{1}}, \gamma_{a_{2}}$. So $\mathcal H_p$ is a directed subset of $pt(L)$.
	
 Since Claim 1 proved that $pt(L)$ is a dcpo, for each $p\in pt(L)$, $\mathcal H_p$ as a directed subset of $pt(L)$ has a sup $\bigvee^{\uparrow}\mathcal H_p$.
	
	$\mathbf{Claim\; 4}$:  For each $p\in pt(L)$,  $p = \bigvee^{\uparrow}\mathcal H_p$.
	
	Clearly, the definition of $\mathcal H_p$ necessitates that $ \bigvee\mathcal H_p\leq p$. See the reverse direction, we just need to consider $x\in L$ which meets $p(x) = 1$. According to the definitions of $FDD$-lattices, there exists a finite subset $F\subseteq Cp(L)$ such that $x = \vee F$. As $p\in pt(L)$, we get that $\vee p(F) = p(\vee F) = 1$, from which we can find an $x_0\in F$ satisfying that $p(x_0) = 1$. This implies that $\gamma_{x_0}\leq p$, then $\gamma_{x_0}\in \mathcal H_p$. Besides, since $x_0\in F$, $x_0\leq x$. This immediately leads to the fact that $\gamma_{x_0}(x) = 1$. Thus there exists a $\gamma_{x_0}\in \mathcal H_p$ such that $p\leq \gamma_{x_0}$. Hence, $p\leq\bigvee\mathcal H_p$.
	
	$\mathbf{Claim\; 5}$: The compact elements of $pt(L)$ are precisely the collection of $\{\gamma_a: a\in Cp(L)\}$.
		
	We first verify the compactness of each $\gamma_a, a\in Cp(L)$. Assume $\{f_i: i\in I\}$ is a directed subset of $pt(L)$ and $\gamma_a\leq \vee^{\uparrow} f_i$. Then we have $\vee^{\uparrow} f_i(a) = \vee^{\uparrow} (f_i(a)) = 1$, which means that there exists an $i_0\in I$ such that $f_{i_0}(a) = 1$, equivalently, $\gamma_a\leq f_{i_0}$. Thus each $\gamma_a, a\in Cp(L)$ is a compact element of $pt(L)$. On the other hand, given an arbitrary compact element $p\in pt(L)$.  By Claim 3 and Claim 4, we know	$\mathcal H_p\subseteq pt(L)$ is directed and $p = \bigvee^{\uparrow}\mathcal H_p$. So by the compactness of $p$, there must be a $\gamma_{a}\in \mathcal H_p$ such that $p\leq\gamma_{a}$. Meanwhile, the definition of $\mathcal H_p$ indicates that $\gamma_{a}\leq p$. Thus $p = \gamma_{a}$. Taken together, $K(pt(L)) = \{\gamma_a: a\in Cp(L)\}$.
	
	Combining Claims 3, 4, and 5, we see that  each $p\in pt(L)$ is actually a supremum of a directed subset of $pt(L)$, in which each element is compact and less than $p$. Thus from the definitions of algebraic domains, we reach that $pt(L)$ is an algebraic domain.
	
	$\mathbf{Claim\;  6}$: $pt(L)$ is an algebraic $L$-domain.
	
	The definition of an $L$-domain requires us to show every subset with an upper bound has a supremum. For the empty subset $\emptyset$ of $pt(L)$,  since $pt(L)$ has a least element $\gamma_1$, the supremum of $\emptyset$  must be $\gamma_1$. For the nonempty subsets of $pt(L)$,
	based on the established fact that $pt(L)$ is an algebraic domain, it suffices  to prove that for any nonempty finite subset $F$ of $K(pt(L))$ with an upper bound $p$ in $pt(L)$, its supremum in ${\downarrow}p$ exists. By Claim 5, we may assume that $F = \{\gamma_{a_i}: a_i\in Cp(L), i = 1,2...n\}$.  For each $i$, $\gamma_{a_i}\leq p$ implies that $p(a_i) = 1$. Since $p\in pt(L)$, $p(\wedge a_i) = \wedge p(a_i) = 1$. Since $L$ is an $FDD$-lattice, we can find a  finite disjoint subset $E\subseteq Cp(L)$ such that $\wedge a_i = \dot\vee E$, which means $p(\dot\vee E) = 1$. By $p\in pt(L)$ again, $\dot\vee \{p(e): e\in E\} = 1$. Then there exists an $e_0\in E$ satisfying $p(e_0) = 1$. Moreover, the existence of $e_0$ is unique. If there is another element $e_1\in E$ satisfying $p(e_1) = 1$, then the equation $p(e_0\wedge e_1) = p(e_0)\wedge p(e_1) = 1$ holds. This will  contradict the fact that $e_0$ and $e_1$ are disjoint.  Thereby, we find a unique $e_0\in Cp(L)$ such that $p(e_0) = 1$, equivalently, there is a $\gamma_{e_0}\leq p$. We next declare that $\gamma_{e_0}$ is just the supremum of $F$. Clearly, by $\wedge a_i = \dot\vee E$, we know $e_0\leq a_i$ for each $i = 1,2...n$. Thus $\gamma_{a_i}\leq \gamma_{e_0}, \forall i = 1,2...n$, that is to say,   $\gamma_{e_0}$ is an upper of $F$. Suppose $h$ is an another upper bound of $F$ in ${\downarrow}p$. Then $h(a_i) = 1$ for each $i$.  Similar to the process described above, we can also find $e'\in E$ such that $h(e') = 1$. Because $h\leq p$, $p(e') = 1$. According to the uniqueness of $e_0$, we have $e_0 = e_1$. This indicates $h(e_0) = 1$, which is equivalent to saying that $\gamma_{e_0}\leq h$. So   $\gamma_{e_0}$ is the least upper bound of $F$, i.e., $\gamma_{e_0} = \vee F$. And hence, $pt(L)$ is an algebraic $L$-domain.
	
	
	$\mathbf{Claim 7}$: $pt(L)$ is a Lawson compact algebraic $L$-domain.
	
	It suffices to prove that $K(pt(L))$ is mub-complete and every finite subset $F$ of which has only finitely minimal upper bounds. We may  set $F = \{\gamma_{b_i}\!: b_i\in Cp(L), i = 1,2...n\}$. Then by Definition \ref{fddl}, we know $\wedge b_i = \dot\vee J$, where $J$ is a finite disjoint subset of $Cp(L)$. This fact yields the following equation:
	\begin{align*}
		F^u \cap K(pt(L)) 
		&= \{f\in pt(L): \gamma_{b_i}\leq f, i = 1,2...n\}\cap K(pt(L))\\
		&= \{f\in pt(L): f(b_i) = 1, i = 1,2...n\} \cap K(pt(L))\\
		&= \{f\in pt(L): f(\wedge b_i) = 1, i = 1,2...n\}\cap K(pt(L))\\
		&= \{f\in pt(L): f(\dot\vee J) = 1\}\cap K(pt(L))\\
		&= \{f\in pt(L): \dot\vee \{f(j): j\in J\} = 1\}\cap K(pt(L))\\
		&= \{f\in pt(L): \exists j\in J\; {\rm{s.t.}}\; f(j) = 1\}\cap K(pt(L))\\
		&= \{f\in pt(L): \exists j\in J\; {\rm{s.t.}}\; \gamma_j\leq f\}\cap K(pt(L))\\
		&={\uparrow}\{\gamma_j: j\in J\}\cap K(pt(L))
	\end{align*}
	Clearly, $\{\gamma_j: j\in J\}$ is a finite subset in $K(pt(L))$. Finally,  we just need to show  that $\{\gamma_j: j\in J\}$ are minimal upper bounds of $F$ in $K(pt(L))$. Pick an arbitrary $j\in J$, suppose that $\gamma_c\in K(pt(L))$ is an upper bound of $F$ and $\gamma_c\leq \gamma_j$.  Since $\gamma_c\in F^u\cap K(pt(L))$, there exists a $j_0\in J$ such that $\gamma_c\geq \gamma_{j_0}$. Then $\gamma_{j_0}\leq \gamma_{j}$, based on the fact that $J$ is disjoint, we get that $\gamma_{j_0} = \gamma_{j}$. This immediately indicates that $\gamma_{c} = \gamma_{j}$. Therefore, each $\gamma_{j}$ is a minimal upper bound of $F$ in $K(pt(L))$.
	\end{proof}
\end{theorem}
\begin{lemma}\label{p10}
	Let $L$ be an $FDD$-lattice. Then for any two elements $a, b\in L$ with $a\nleq b$, there exists a point $p\in pt(L)$ such that $p(a) = 1, p(b) = 0$.
	
	\begin{proof}
		As $L$ is an $FDD$-lattice, one can find a finite subset $F\subseteq Cp(L)$ such that $a = \vee F$, which means $\vee F\nleq b$. So  there is an $a_0\in F$ satisfying that $a_0\nleq b$. Let us define $\gamma_{a_0}: L\rightarrow \{0, 1\}$:
		\vspace{-1.5em}
		\begin{center}
		\[ \forall x\in L,
		\gamma_{a_0}(x) = 
			\begin{cases}
				1, & \text{if} \;\;  x \geq a_0;  \\
				0, &  \text{otherwise}.
			\end{cases}
			\]
		\end{center}
		By Claim 2 in Theorem \ref{th37} and the fact that $a_0\in Cp(L)$, we know $\gamma_{a_0}\in pt(L)$. Since $a_0\leq a$ and $a_0\nleq b$, we have $\gamma_{a_0}(a) = 1, \gamma_{a_0}(b) = 0$.
	\end{proof}
\end{lemma}

\begin{theorem}\label{eqa1}
	Let $L$ be an $FDD$-lattice. Then $L$ is represented by the collection of the compact open subsets of the Lawson compact algebraic $L$-domain $pt(L)$.
	\begin{proof}
		We define a mapping $\eta_L: L\rightarrow CO(pt(L))$ as:
		\begin{center}
			$\eta_L(x) = \{p\in pt(L): p(x) = 1\}, \forall x\in L$.
		\end{center}
		
	$\mathbf{Claim 1}$: $\eta_L$ is well-defined.
	
	For any $x\in L$, we can find a finite subset $F\subseteq Cp(L)$ satisfying that $x = \vee F$ since $L$ is an $FDD$-lattice. Then for each $p\in \eta_L(x)$, $p(\vee F) = 1$. Thus there exists an $a_p\in F$ such that $p(a_p) = 1$ by  $p\in pt(L)$,  which immediately indicates that $p\geq \gamma_{a_p}$. Hence, $\eta_L(x) \subseteq {\uparrow}\{\gamma_a: a\in F\}$. Conversely, for any $a\in F$, $a\leq x$, thus $\gamma_a(x) = 1$. This means that each $\gamma_a$ belongs to $\eta_L(x)$. So $\eta_L(x) = {\uparrow}\{\gamma_a: a\in F\}$, where ${\uparrow}\{\gamma_a: a\in F\}$ is clearly a compact-open subset of $pt(L)$.
	
	$\mathbf{Claim 2}$: $\eta_L$ is order-preserving.
	
	For any $x, y\in L$ with $x\leq y$, we need to prove that $\eta_L(x)\subseteq \eta_L(y)$. Assume $p\in \eta_L(x)$, that is, $p\in pt(L)$ and $p(x) = 1$. Since $x\leq y$, $p(y) = 1$. Consequently, $p\in \eta_L(y)$. Thus $\eta_L(x)\subseteq \eta_L(y)$ holds.
	
		$\mathbf{Claim 3}$: $\eta_L$ is injective.
		
		For any $x, y\in L$ with $x\neq y$, without loss of generality,  assume that $x\nleq y$. Then by Lemma \ref{p10}, we can find a point $p\in pt(L)$ such that $p(x) = 1$ and $p(y) = 0$. From which we get that $p\in \eta_L(x)$ and $p\notin \eta_L(y)$, so $\eta_L(x)\neq\eta_L(y)$.

	$\mathbf{Claim 4}$: $\eta_L$ is surjective.
	
	Let $U$ be an arbitrary element of $CO(pt(L))$.  Then by lemma \ref{co} and Claim 5 in Theorem \ref{th37}, we know $U = \cup\{{\uparrow}\gamma_{a_i}: a_i\in Cp(L), i\in F\}$, where  $F$ is a finite subset of the natural numbers. A further assumption we can make is that none of these sets $\{{\uparrow}\gamma_{a_i}: i\in F\}$ is a subset of another. Next we prove that $U = \eta_L(\vee_{i\in F}a_i)$, where $\eta_L(\vee_{i\in F}a_i) = \{p\in pt(L): p(\vee_{i\in F}a_i) = 1\}$. On the one hand, for each $p\in U$, there exists an $i_0\in F$ such that $p\in {\uparrow}\gamma_{a_{i_0}}$, which means $p(a_{i_0}) = 1$. It follows that $p(\vee_{i\in F}a_i) = 1$. So $p\in \eta_L(\vee_{i\in F}a_i)$. On the other hand, pick $p\in \eta_L(\vee_{i\in F}a_i)$, then $p(\vee_{i\in F}a_i) = 1$. Since $p$ is an $FDD$-lattice homomorphism, we have $\vee\{p(a_i): i\in F\} = p(\vee_{i\in F}a_i) = 1$. This in turn implies that $p(a_{i_0}) = 1$ for some $i_0\in F$. Meanwhile, $p(a_{i_0}) = 1$ indicates that $p\geq \gamma_{a_{i_0}}$, that is to say, $p\in U$. Consequently,  $U = \eta_L(\vee_{i\in F}a_i)$. Hence, $\eta_L$ is a surjective mapping.
		
	\end{proof}
	\end{theorem}

\begin{lemma}\label{lca}
	Let $L, M$ be two $FDD$-lattices and $f: L\rightarrow M$ a lattice homomorphism. Then $pt(f): pt(M)\rightarrow pt(L)$ is a spectral mapping, where $pt(f)$ is defined as following:
	\begin{center}
	 $pt(f)(p) = p\circ f, \forall p\in pt(M)$.
	\end{center}
	It is trivial to show that $p\circ f\in pt(L)$ as $p$ and $f$ are both lattice homomorphisms.

\end{lemma}

\begin{lemma}\label{fdd}
	Let $L, M$ be two pointed Lawson compact algebraic $L$-domains and $f: L\rightarrow M$ a spectral mapping. Then $CO(f): CO(M)\rightarrow CO(L)$  defined by
	 	\begin{center}
	 	$CO(f)(U) = f^{-1}(U), \forall U\in CO(M)$.
	 \end{center}
	 is a lattice homomorphism.
	 \begin{proof}
	 	It is straightforward  to verify.
	 \end{proof}
\end{lemma}

%

\begin{theorem}
	The categories $\mathbf{FFD}$ and $\bf{LCA}$ are dually equivalent.
	\begin{proof}
		We set $pt: \mathbf{FFD}\rightarrow \bf{LCA}$ that assigns to an $FDD$-lattice  the Lawson compact algebraic $L$-domain $pt(L)$, and each $FDD$-lattice homomorphism $f: L\rightarrow M$ the spectral mapping $pt(f): pt(M)\rightarrow pt(L)$ defined as  in Lemma \ref{lca}; Set $CO:  \bf{LCA}\rightarrow \bf{FFD}$ that associates with a Lawson compact algebraic $L$-domain $S$ the $FDD$-lattice $CO(S)$ and a spectral mapping $g: S\rightarrow T$ the $FDD$-lattice homomorphism $CO(g): CO(T)\rightarrow CO(S)$ defined as  in Lemma \ref{fdd}.
		
		Let $f: L\rightarrow M$ and $g: M\rightarrow N$ be two morphisms in $\mathbf{FFD}$. Then for an arbitrary point $p\in pt(N)$,
		\begin{align*}
			pt(g\circ f)(p)  &= p\circ(g\circ f) \\
			&= (p\circ g)\circ f\\
			&= pt(f)(p\circ g)\\
			&= pt(f)\circ pt(g)(p).
		\end{align*}
This reveals that $pt$ preserves the compositions of morphisms in $\mathbf{FFD}$. Similarly, one can easily verify that $pt$ preserves identity morphisms in $\mathbf{FFD}$. Thus $pt$ is a functor.

Suppose $h: S\rightarrow T$ and $k: T\rightarrow Y$ are two morphisms in $\bf{LCA}$. Then for any compact-open subset $U$ of $Y$, 
	\begin{align*}
	CO(k\circ h)(U)  &= (k\circ h)^{-1}(U) \\
	&= h^{-1}(k^{-1}(U))\\
	&= CO(h) (CO(k)(U))\\
	&= CO(h)\circ CO(k)(U).
\end{align*}
	 Thus $CO$ preserves the compositions of morphisms in $\bf{LCA}$. It is  trivial to prove that $CO$ preserves the identity morphisms in $\bf{LCA}$. So, $CO$ is a functor.
		
	To prove the equivalence between the two categories $\bf{FFD}$ and $\bf{LCA}$, we proceed to define two functions: one is  $\eta: Id\rightarrow CO\circ pt$ between the identity functor and the $CO\circ pt$ functor, the other is  $\theta: Id\rightarrow pt\circ CO$ between the identity functor and the $pt\circ CO$ functor, in which for any $FDD$-lattice $L$ and any Lawson compact algebraic $L$-domain $M$, $\eta_L$ and $\theta_M$ are given as in Theorem \ref{eqa1} and \ref{plo81}, respectively. Besides, we have proved that  $\eta_L$ and $\theta_M$ are both isomorphisms. Thus, we just need to show that $\eta$ and $\theta$ are natural transformations, that is, for each arrow $f:L\rightarrow M$ in $\bf{FFD}$ and each arrow $g:S\rightarrow T$ in $\bf{LCA}$, the following two diagrams commute.
\begin{figure}[h]
	\centering
	\begin{tikzpicture}[line width=0.6pt,scale=0.7]
		\node[right][font=\scriptsize]  at(-5.5,1)(d)  {\large $L$};
		
		\node[right][font=\scriptsize]  at(-1,1) (f)  {\large $CO\circ pt(L)$};
		\node[right][font=\scriptsize]  at(-5.5,-1.5) (g)  {\large $M$};
		\node[right][font=\scriptsize]  at(-1,-1.5) (h)  {\large $CO\circ pt(M)$};
		\draw[->](d)--(f);
		
		\draw[->](d)--(g);
		\draw[->](f)--(h);
		\draw[->](g)--(h);
		\node[right][font=\scriptsize]  at(-3.2,1.3)  {\large $\eta_L$};
		\node[right][font=\scriptsize]  at(-3.2,-1.2)  {\large $\eta_M$};
		\node[right][font=\scriptsize]  at(-5.1,-0.3)  {\large $f$};
		\node[right][font=\scriptsize]  at(0.5,-0.3)  {\large $CO\circ pt(f)$};
		
			\node[right][font=\scriptsize]  at(4,1)(d)  {\large $S$};
		\node[right][font=\scriptsize]  at(8.5,1) (f)  {\large $pt\circ CO(S)$};
		\node[right][font=\scriptsize]  at(4,-1.5) (g)  {\large $T$};
		\node[right][font=\scriptsize]  at(8.5,-1.5) (h)  {\large $pt\circ CO(T)$};
		\draw[->](d)--(f);
		
		\draw[->](d)--(g);
		\draw[->](f)--(h);
		\draw[->](g)--(h);
		\node[right][font=\scriptsize]  at(6,1.3)  {\large $\theta_S$};
		\node[right][font=\scriptsize]  at(6,-1.2)  {\large $\theta_T$};
		\node[right][font=\scriptsize]  at(4.3,-0.3)  {\large $g$};
		\node[right][font=\scriptsize]  at(10.1,-0.3)  {\large $pt\circ CO(g)$};
		
	\end{tikzpicture}
\end{figure}
	
	For each $x\in L$, $\eta_M\circ f(x) = \{p\in pt(M): p(f(x)) = 1\}$. Meanwhile, 
		\begin{align*}
		CO\circ pt(f)\circ\eta_L(x)  &= pt(f)^{-1}(\eta_L(x)) \\
		&=\{p\in pt(M): pt(f)(p)\in \eta_L(x)\}\\
		&= \{p\in pt(M): p\circ f\in \eta_L(x)\}\\
		&= \{p\in pt(M): p\circ f(x) = 1\}.
	\end{align*}
Then directly from which we can draw the conclusion that $\eta_M\circ f = 	CO\circ pt(f)\circ\eta_L$, i.e. the first diagram commutes.
	
	To verify the commutativity of the second diagram, pick an arbitrary element $s\in S$, $\theta_T\circ g(s) =  \vee^{\uparrow}\{\gamma_{{\uparrow}k}: k\in {\downarrow}g(s)\cap K(T)\}$; $	pt\circ	CO (g)\circ\theta_S(s)  = pt\circ CO(g)(\vee\{\gamma_{{\uparrow}l}: l\in {\downarrow}s	\cap K(S)\})$. In order to prove these two points of $CO(T)$ are identical with each other, assume $U\in CO(T)$, then to show that  $\vee^{\uparrow}\{\gamma_{{\uparrow}k}: k\in {\downarrow}g(s)\cap K(T)\}(U) = pt\circ CO(g)(\vee^{\uparrow}\{\gamma_{{\uparrow}l}: l\in {\downarrow}s	\cap K(S)\})(U)$, in which $\vee^{\uparrow}\{\gamma_{{\uparrow}k}: k\in {\downarrow}g(s)\cap K(T)\}(U) = \vee^{\uparrow}\{\gamma_{{\uparrow}k}(U): k\in {\downarrow}g(s)\cap K(T)\}$ and
		\begin{align*}
pt\circ CO(g)(\vee^{\uparrow}\{\gamma_{{\uparrow}l}: l\in {\downarrow}s\cap K(S)\})(U) &= \vee^{\uparrow}\{\gamma_{{\uparrow}l}: l\in {\downarrow}s\cap K(S)\}\circ CO(g)(U)\\
	&=\vee^{\uparrow}\{\gamma_{{\uparrow}l} (g^{-1}(U)): l\in {\downarrow}s\cap K(S)\}.
	\end{align*}
Thus it suffices to prove that	$\vee^{\uparrow}\{\gamma_{{\uparrow}k}(U): k\in {\downarrow}g(s)\cap K(T)\}= \vee^{\uparrow}\{\gamma_{{\uparrow}l} (g^{-1}(U)): l\in {\downarrow}s\cap K(S)\}$.  If $\vee^{\uparrow}\{\gamma_{{\uparrow}k}(U): k\in {\downarrow}g(s)\cap K(T)\} = 1$, then there exists a $k_0\in {\downarrow}g(s)\cap K(T)$ such that $\gamma_{{\uparrow}k_0}(U) = 1$, which immediately indicates that ${\uparrow}k_0\subseteq U$, that is, $k_0\in U$. Since $k_0\leq g(s)$ and $U$ is an upper bound, $g(s)\in U$, i.e. $s\in g^{-1}(U)$.
 From the facts that  $S$ is a Lawson compact algebraic $L$-domain and $g^{-1}(U)$ is a Scott open subset of $S$, we deduce that  $s = \vee^{\uparrow}({\downarrow}s\cap K(S))\in g^{-1}(U)$. So we can find an $l_0\in {\downarrow}s\cap K(S)$ that belongs to $g^{-1}(U)$, this means  ${\uparrow}l_0\subseteq g^{-1}(U)$ and furthermore, $\gamma_{{\uparrow}l_0}(g^{-1}(U)) = 1$. Hence, $\vee^{\uparrow}\{\gamma_{{\uparrow}l} (g^{-1}(U)): l\in {\downarrow}s\cap K(S)\} = 1$.
 
  If $\vee^{\uparrow}\{\gamma_{{\uparrow}k}(U): k\in {\downarrow}g(s)\cap K(T)\} = 0$, then for all $k\in {\downarrow}g(s)\cap K(T)$, $\gamma_{{\uparrow}k}(U) = 0$. Directly from which we get that ${\uparrow}k\nsubseteq U$, i.e. $k\notin U$ for all $k\in {\downarrow}g(s)\cap K(T)$. Now suppose that there is an $l'\in {\downarrow}s\cap K(S)$ such that $\gamma_{{\uparrow}l'} (g^{-1}(U)) = 1$, that is to say that ${\uparrow}l'\subseteq g^{-1}(U)$, and moreover, $g(l')\in U$. Because $T$ is Lawson compact algebraic $L$-domain, we have that $g(l') = \vee^{\uparrow}({\downarrow}g(l')\cap K(T))\in U$.  By the Scott openness of $U$, there exists a $k_0\in {\downarrow}g(l')\cap K(T)$ satisfying $k_0\in U$. Note that $l'\leq s$, so $g(l')\leq g(s)$. Thus, we find a $k_0\in {\downarrow}g(s)\cap K(T)$ belonging to $U$. This obviously contradicts the result proven earlier that no   $k\in {\downarrow}g(s)\cap K(T)$ is in $U$. So, for any $l\in {\downarrow}s\cap K(S)$, $\gamma_{{\uparrow}l} (g^{-1}(U)) = 0$. This consequently means that $\vee^{\uparrow}\{\gamma_{{\uparrow}l} (g^{-1}(U)): l\in {\downarrow}s\cap K(S)\} = 0$. 
  
  As a conclusion, $\vee^{\uparrow}\{\gamma_{{\uparrow}k}(U): k\in {\downarrow}g(s)\cap K(T)\}= \vee^{\uparrow}\{\gamma_{{\uparrow}l} (g^{-1}(U)): l\in {\downarrow}s\cap K(S)\}$ holds, in other words, the second diagram commutes. And hence, the categories $\bf{FFD}$ and $\bf{LCA}$ are dually equivalent.
	\end{proof}
\end{theorem}

\bibliographystyle{plain}

\end{document}